\documentclass[11pt,a4paper, reqno, centertags]{amsart}

\usepackage{indentfirst} 
\usepackage{amsmath, amsthm, amssymb,xfrac, mathrsfs} 
\usepackage{BOONDOX-calo, graphicx, accents}
\usepackage{tikz} 
\usetikzlibrary{hobby}
\usetikzlibrary{decorations.pathreplacing}
\usepackage{wrapfig} 
\usepackage{xcolor} 
\usepackage{verbatim} 
\usepackage{enumerate} 

\usepackage{hyperref} 
\hypersetup{
colorlinks=true,
citecolor=black!50!red,
linkcolor=black!50!red,
linktoc=all
}
\usepackage[all]{hypcap} 

\newcounter{constant}

\newcounter{bigconstant}

\newtheorem{theorem}{Theorem}[section]
\newtheorem{proposition}[theorem]{Proposition}
\newtheorem{lemma}[theorem]{Lemma}

\newtheorem{corollary}[theorem]{Corollary}
\newtheorem*{paley}{Paley-Wiener Theorem}

\numberwithin{equation}{section} 

\theoremstyle{definition}

\newtheorem{remark}[theorem]{Remark}

\newcommand{\R}{\mathbb{R}}
\newcommand{\T}{\mathbb{T}}

\newcommand{\N}{\mathbb{N}}

\makeatother

\begin{document}
    \title[Improved algebraic lower bound]{ Improved algebraic lower bound for the radius of spatial analyticity for the generalized KdV equation}
    \author[M. Baldasso and M. Panthee]{ Mikaela Baldasso and Mahendra Panthee}
    \address{Department of Mathematics, University of Campinas\\13083-859 Campinas, SP, Brazil}
\email{mikaelabaldasso@gmail.com, mpanthee@unicamp.br}
\thanks{This work was partially supported by CAPES, CNPq and FAPESP, Brazil.}
\maketitle

\begin{abstract}
    We consider the initial value problema (IVP) for the generalized Korteweg-de Vries (gKdV) equation
\begin{eqnarray*}
\begin{cases}
\partial_{t}u+\partial_{x}^{3}u+\mu u^{k}\partial_{x}u=0, \, &x\in \R, \, t \in \R,\\
u(x,0)=u_0(x),
\end{cases}
\end{eqnarray*}
where $u(x,\,t)$ is a real valued function, $u_0(x)$ is a real analytic function, $\mu=\pm 1$ and $k\geq 4$. We prove that if the initial data $u_0$ has radius of analyticity $\sigma_0$, than there exists $T_0>0$ such that the radius of spatial analyticity of the solution remains the same in the time interval $[-T_0, \, T_0]$. In the defocusing case, for $k\geq 4$ even, we prove that when the local solution extends globally in time, then for any $T\geq T_0$, the radius of analyticity cannot decay faster than $cT^{-\left(\frac{2k}{k+4}+\epsilon\right)}$, $\epsilon>0$  arbitrarily small and $c>0$ a constant. The result of this work improves the one obtained by Bona et al. in \cite{bona2005algebraic}.
\end{abstract}



\section{Introduction}\label{s:intro}

Consider the initial value problem (IVP) for the generalized Korteweg-de Vries (gKdV) equation:
\begin{eqnarray}\label{eq:gKdV}
\begin{cases}
\partial_{t}u+\partial_{x}^{3}u+\mu u^{k}\partial_{x}u=0, \, &x\in \R, \, t \in \R,\\
u(x,0)=u_0(x),
\end{cases}
\end{eqnarray}
where the unknown $u(x,t)$ and the initial data $u_0(x)$ are real-valued, $\mu=\pm 1$ and $k \geq 4$. When $\mu=1$, the equation \eqref{eq:gKdV} is known as focusing and defocusing for $\mu=-1$.

It is well-known that the real solution for the IVP \eqref{eq:gKdV} enjoys the mass conservation
\begin{equation}\label{eq:mass conversation}
    M(u(t))=\int_{\R} u^2(x,\,t)dx=M(u_0),
\end{equation}
and the energy conservation
\begin{equation}\label{eq:energy conservation}
    E(u(t))=\int_{\R}(\partial_x u)^2(x,\,t)dx -\frac{2\mu}{(k+1)(k+2)}\int_{\R}u^{k+2}(x,\,t)dx,
\end{equation}
and these quantities will be useful in this work.

Well-posedness issues of the IVP \eqref{eq:gKdV} for $k\geq 1$ with given initial data in the classical Sobolev spaces $H^s(\R)$ have been extensively studied by several authors, see \cite{colliander2003sharp}, \cite{farah2010global}, \cite{farah2010supercritical},  \cite{kenig1996bilinear}, \cite{kenig1993well} and references therein. Particularly, for $k=4$, Kenig, Ponce and Vega \cite{kenig1993well} proved the local well-posedness in $H^s(\R)$ for $s \geq 0$ which yields global solutions for small data. The global well-posedness for given data with Sobolev norm not so small was proved by Fonseca, Linares and Ponce  \cite{fonseca2003global}  in $H^s(\R)$, $s>3/4$. In \cite{farah2010global}, Farah improved the latter result and lowered the index to $s>3/5$ and posteriorly, the authors in \cite{miao2009low} extended the global well-posedness for $s>6/13$.  For $k\geq 5$, the best local well-posedness known results in $H^s(\R)$ are for $s>\frac{1}{2}-\frac{2}{k}$, which attain the critical scaling indices. In the defocusing case ($\mu=-1$) and $k$ even, Farah, Linares and Pastor, in \cite{farah2010supercritical}, established the global well-posedness in $H^s(\R)$ for $s>\frac{4(k-1)}{5k}$. For the focusing case ($\mu=1$) or the defocusing case ($\mu=-1$) with $k$ odd, the same authors also proved the global well-posedness in $H^1(\R)$ for sufficiently small data.

In this work we are interested in studying the well-posedness of the IVP \eqref{eq:gKdV} with real analytic initial data $u_0$. For this purpose, we consider $u_0$ in the Gevrey space $G^{\sigma,\,s}(\R)$, $\sigma>0$ and $s\in \R$, defined as
\begin{equation*}
    G^{\sigma,\,s}(\R)=\left\lbrace f \in L^2(\R):\, ||f||_{G^{\sigma,\,s}(\R)}^2=\int_{\R} \langle \xi\rangle^{2s}e^{2\sigma |\xi|}|\hat{f}(\xi)|^2d\xi <\infty\right\rbrace,
\end{equation*}
where $\hat{f}$ denotes the spatial Fourier transform of $f$,
\begin{equation*}
    \hat{f}(\xi)=\frac{1}{\sqrt{2\pi}}\int_{\R}e^{-ix\xi}f(x)dx.
\end{equation*}
The interest in these spaces is due to the following fact.

\begin{paley} Let $\sigma>0$ and $s \in \R$. Then, the following are equivalent:
\begin{enumerate}
    \item $f \in G^{\sigma,\,s}$.
    \item $f$ is the restriction to the real line of a function $F$ which is holomorphic in the strip
    \begin{equation*}
        S_{\sigma} =\{x+iy:\, x,\, y\in \R, \, |y|<\sigma\}
    \end{equation*}    
    and satisfies
    \begin{equation*}
        \sup_{|y|<\sigma}||F(x+iy)||_{H_x^s}<\infty.
    \end{equation*}
\end{enumerate}
\end{paley}
In this sense, $\sigma$ is called the uniform radius of analyticity and a natural question concerning this space is: given $u_0 \in G^{\sigma,\,s}$ is it possible to guarantee the existence of solution such that the radius of analyticity remains the same at least for short time? And after extending the local solution globally in time, how does the radius of analyticity evolve in time? This sort of questions for the dispersive equations are widely studied in the literature, see for example \cite{barostichi2021modified}, \cite{BHK}, \cite{bona2005algebraic}, \cite{figueira2023decay}, 
\cite{grujic2002}, \cite{hayashi}, \cite{kato1986}, \cite{selberg2015lower}, \cite{selberg2019radius}, \cite{selberg2017radius} and references therein. Se also \cite{HHP}, \cite{HKS}, \cite{HG}, \cite{QL} and references therein for the problems posed in the periodic domain $\T$.

Concerning the IVP \eqref{eq:gKdV} for $k\geq 1$ with data in the Gevrey spaces $G^{\sigma, \, s}(\R)$, we mention the work of Bona et al. in \cite{bona2005algebraic} where the authors proved the global well-posedness results and obtained algebraic lower bounds for the radius of analyticity as time advances. More precisely, for the KdV ($k=1$) and the mKdV ($k=2$) equations, they obtained $cT^{-12}$ as an algebraic lower bound for the evolution of the radius of analyticity and $cT^{-(k^2+3k+2)}$ for $k\geq 3$. Recently, in \cite{selberg2015lower}, Selberg and Silva introduced the concept of almost conserved quantity and improved the result in \cite{bona2005algebraic} obtaining $cT^{-(\frac{4}{3}+\epsilon)}$ as a lower bound for the radius of analyticity for the KdV equation. For the gKdV equation with $k=3$, Selberg and Tesfahun obtained, in \cite{selberg2017radius}, $cT^{-2}$ as a lower bound for the radius of spatial analyticity. In both the KdV equation  and the gKdV equation with $k=3$, almost conserved quantities at the $L^2$-level of Sobolev regularity are used.  For the mKdV equation $(k=2)$ no well-posedness result exists at the $L^2$-level of Sobolev regularity. This suggests that, if one wishes to use the idea of almost conserved quantity it is necessary to work in the higher level of Sobolev regularity. Quite recently,  the second author with  Figueira, in \cite{figueira2023decay}, constructed almost conserved quantity at the $H^1$-level of Sobolev regularity and obtained the lower bound $cT^{-\frac{4}{3}}$ for the radius of analyticity to the mKdV equation improving the result in \cite{bona2005algebraic}. Now, a natural question that arise is that whether  it is possible to construct almost conserved quantity for the gKdV equation with $k\geq 4$ so as to improve the lower bound for the radius of analyticity obtained in \cite{bona2005algebraic}.

In the present work, taking motivation from \cite{figueira2023decay}, we will respond the question raised in the previous paragraph affirmatively by constructing almost conserved quantity at the $H^1$-level of Sobolev regularity (see \eqref{eq:conserved quantity} and \eqref{eq:estimate energy} below) and improve significantly the result  in \cite{bona2005algebraic} for the defocusing gKdV equation \eqref{eq:gKdV} with $k\geq 4$ even.

 Regarding the local well-posedness, we first prove the following  result that is valid for both focusing and defocusing cases and for every integer $k\geq 4$.

\begin{theorem}\label{teo:local}
    Let $\sigma>0$, $k\geq 4$ and $s>\frac{k-4}{2k}$. For given $u_0 \in G^{\sigma,\, s}(\R)$, there exists $T_0=T_0(||u_0||_{G^{\sigma,\,s}})>0$ such that the IVP \eqref{eq:gKdV} admits a unique solution $u \in C([-T_0,\,T_0]; G^{\sigma,\, s}(\R))$. Moreover, the data-to-solution map is locally Lipschitz.
\end{theorem}

Thus, for short time, the solution remains analytic in the same initial strip. Our main result concerning the global well-posedness and the evolution of the radius of analyticity is the following.

\begin{theorem}\label{teo:global}
    Let $\sigma_0>0$, $k\geq 4$ even, $s>\frac{k-4}{2k}$, $u_0 \in G^{\sigma,\, s}(\R)$ and $u$ be the local solution to the IVP \eqref{eq:gKdV} in the defocusing case ($\mu=-1$) given by Theorem \ref{teo:local}. Then, for any $T\geq T_0$, the local solution $u$ extends globally in time satisfying
    \begin{equation*}
        u\in C([-T,\,T];G^{\sigma(T),\, s}(\R)), \quad \text{with } \sigma(T) \geq \min \left\lbrace\sigma_0 , \, cT^{-\left(\frac{2k}{k+4}+\epsilon \right)}\right\rbrace,
    \end{equation*}
where $\epsilon>0$ is arbitrarily small and $c$ is a positive constant depending on $k$, $s$, $\sigma_0$ and $||u_0||_{G^{\sigma,\, s}}$.
\end{theorem}

To prove Theorem \ref{teo:global}, we derive an almost conserved quantity in $G^{\sigma,\,1}(\R)$ space using the conserved quantities \eqref{eq:mass conversation} and \eqref{eq:energy conservation}. Once having the almost conserved quantity at hand, we are able to prove the global result by decomposing any interval of time $[0,\,T]$ into short subintervals and iterate the local result in each subinterval. During this iteration process there appears a restriction that gives the lower bound for $\sigma(T)$.

This paper is organized as follows. In Section \ref{sec:2} we define the Bourgain's spaces and record some preliminary estimates. The proof of the local well-posedness result stated in Theorem \ref{teo:local} is contained in Section \ref{sec:3}. In Section \ref{sec:4} we derive the almost conserved quantity and find the associated decay estimates. Finally, in Section \ref{sec:5} we extend the local solution globally in time and obtain an algebraic lower bound for the radius of analyticity stated in Theorem \ref{teo:global}.

\section{Function Spaces and k-linear estimates}\label{sec:2}
In this section we discuss the function spaces that will be used throughout this work and derive some k-linear estimates that play crucial role in the proofs.

First, regarding the Gevrey space defined in Section \ref{s:intro}, we have the embedding
\begin{equation}\label{eq:Gegrey embedding}
    G^{\sigma,\,s}\subset G^{\sigma',\, s'} \quad \text{for all } 0<\sigma '<\sigma \text{ and } s, \, s' \in \R,
\end{equation}
and the inclusion is continuous in the sense that there exists a constant $C>0$ depending on $\sigma,\, \sigma ',\, s,\, s'$ such that
\begin{equation*}
    ||f||_{G^{\sigma ',\,s'}}\leq C||f||_{G^{\sigma,\,s}}.
\end{equation*}

In addition to the Gevrey space, we use a space that is a mix between the Gevrey space and the Bourgain's space introduced in \cite{bourgain1} and \cite{bourgain2}. Given $\sigma\geq 0$ and $s,\, b \in \R$, we define the Gevrey-Bourgain space, denoted by $X^{\sigma,\, s,\,b}(\R^2)$, as the closure of the Schwartz space under the norm
\begin{equation*}
||u||_{X^{\sigma,\, s,\,b}}=||e^{\sigma |\xi|}\langle \xi \rangle^{s} \langle \tau - \xi^{3} \rangle ^{b} \tilde{u}(\xi, \tau)||_{L^{2}_{\tau,\, \xi}},
\end{equation*} 
where $\langle \xi \rangle = 1+|\xi|$ and $\tilde{u}$ denotes the space-time Fourier transform of $u$. For $\sigma=0$, we recover the classical Bourgain's space  $X^{s,\,b}(\R^2)$ with the norm given by
\begin{equation*}
||u||_{X^{s,\,b}}=||\langle \xi \rangle^{s} \langle \tau - \xi^{3} \rangle ^{b} \tilde{u}(\xi, \tau)||_{L^{2}_{\tau,\, \xi}}.
\end{equation*}

For $T>0$ the restrictions of $X^{ s,\,b}(\R^2)$ and $X^{\sigma, \, s,\,b}(\R^2)$ to a time slab $\R \times (-T,\, T)$, denoted by $X_{T}^{ s,\,b}(\R^2)$ and $X_{T}^{\sigma, \, s,\,b}(\R^2)$, respectively, are Banach spaces when equipped with the norms
\begin{align*}
||u||_{X_{T}^{s,\,b}}&=\text{inf}\{||v||_{X^{s,\,b}}: v=u \text{ on } \R \times (-T,\, T)\},\\
||u||_{X_{T}^{\sigma, \,s,\,b}}&=\text{inf}\{||v||_{X^{\sigma, \, s,\,b}}: v=u \text{ on } \R \times (-T,\, T)\}.
\end{align*}

To simplify the exposition we introduce the operator $e^{\sigma |D_x|}$ given by
\begin{equation*}
\widehat{e^{\sigma |D_x|}u}(\xi)=e^{\sigma |\xi|}\hat{u}(\xi)
\end{equation*}
so that, one has
\begin{equation}\label{eq:exponential}
||e^{\sigma |D_x|}u||_{X^{s,\, b}}=||u||_{X^{\sigma,\,s,\, b}}.
\end{equation}
Substituting $u$ by $e^{\sigma |D_x|}u$, the relation \eqref{eq:exponential} allows us to carry out the properties of $X^{s, \, b}$ and $X_T^{s, \, b}$ spaces over $X^{\sigma,\,s,\, b}$ and $X_T^{\sigma,\,s,\, b}$ spaces.

Now we record some useful results that will be used in this  work. In the case $\sigma=0$, for the proof of the first lemma below we refer to Section 2.6 of \cite{tao2006nonlinear} and the second lemma follows by the argument used to prove Lemma 7 in \cite{selberg2019radius}. The proofs for $\sigma>0$ follow analogously using the relation \eqref{eq:exponential}.

\begin{lemma}Let $\sigma \geq 0$, $s \in \R$ and $b>\frac{1}{2}$. Then, $X^{\sigma,\,s,\, b}(\R^2)\subset C(\R,\, G^{\sigma,\, s})$ and
\begin{equation*}
\sup_{t \in \R} ||u(t)||_{G^{\sigma, \, s}} \leq C ||u||_{X^{\sigma,\,s,\, b}},
\end{equation*}
where the constant $C>0$ depends only on $b$.
\end{lemma}

\begin{lemma}\label{lemma:restriction} Let $\sigma \geq 0$, $s \in \R$ and $-\frac{1}{2}<b<\frac{1}{2}$ e $T>0$. Then, for any time interval $I \subset [-T, \, T]$, we have 
\begin{equation*}
||\chi_{I}u||_{X^{\sigma,\,s,\, b}} \leq C ||u||_{X_{T}^{\sigma,\,s,\, b}},
\end{equation*}
where $\chi_{I}$ is the characteristic function of $I$ and $C>0$ depends only on $b$.
\end{lemma}

In what follows, let $\psi \in C_0^{\infty}(-2,\,2)$ be a cut-off function with $0\leq \psi \leq 1$, $\psi(t)=1$ on $[-1,\,1]$ and $\psi_T(t)=\psi\left(\frac{t}{T}\right)$ for $T>0$.

Consider the following IVP, for given $F(x,t)$ and $u_0(x)$,
\begin{eqnarray}\label{eq:Cauchy problem}
\begin{cases}
\partial_{t}u+\partial_{x}^{3}u=F, \\
u(x,0)=u_0(x).
\end{cases}
\end{eqnarray}
Using the Duhamel's formula we may write the IVP \eqref{eq:Cauchy problem} in its equivalent integral equation form as
\begin{equation*}
u(t)=W(t)u_0-\int_0^t W(t-t') F(t')dt',
\end{equation*}
where $W(t)=e^{-t\partial_{x}^{3}}=e^{itD_x^{3}}$ is the semigroup associated to the linear problem. The semigroup $W(t)$ satisfies the following estimates in the $X^{\sigma,\,s,\, b}$ spaces. For a detailed proof we refer to \cite{bona2005algebraic}, \cite{ginibre1995probleme} and references therein.

\begin{lemma}\label{lemma: semigroup estimates}Let $\sigma \geq 0$, $s \in \R$ and $\frac{1}{2}<b<b'<1$. Then, for all $0 <T\leq 1$, there is a constant $C=C(s,b)$ such that
\begin{equation}\label{eq:semigroup 1}
||\psi(t)W(t)f(x)||_{X^{\sigma,\,s,\, b}} \leq C||f||_{G^{\sigma, \, s}},
\end{equation}
and
\begin{equation}\label{eq:semigroup 2}
\left|\left| \psi_T(t) \int_{0}^{t}W(t-t')f(x,t')dt'\right|\right|_{X^{\sigma,\,s,\, b}} \leq CT^{b'-b}||f||_{X^{\sigma,\,s,\, b'-1}}.
\end{equation}
\end{lemma}

Now, we derive the k-linear estimates that are key for proving the main results of this work. We start with the following estimate.

\begin{lemma} \label{lemma:estimates partial} Let $k \in \N$, $k \geq 4$ and $s>\frac{k-4}{2k}$. Then, for all $0 <\epsilon \ll 1$, we have
\begin{equation}\label{eq:estimate partial derivative}
||\partial_x(u_1 \cdots u_{k+1})||_{X^{s,-\frac{1}{2}+2\epsilon}} \leq C\prod_{i=1}^{k+1}||u_{i}||_{X^{s,\frac{1}{2}+\epsilon}},
\end{equation}
where $C>0$ depends on $s$, $k$ and $\epsilon$.
\end{lemma}
\begin{proof}For the proof, we refer \cite{miao2009low} for $k=4$ and \cite{farah2010supercritical} for $k\geq 5$. In both references, the authors prove the estimate \eqref{eq:estimate partial derivative} for the $k^{\text{th}}$ power of the function $u$, thas is, for $\partial_x(u^{k+1})$. This new case presented here can be proved with some obvious modifications using the Leibniz rule repeatedly.
\end{proof}

\begin{proposition}\label{prop: estimates} Let  $\sigma \geq 0$, $k \in \N$, $k \geq 4$ and $s>\frac{k-4}{2k}$ be given.Then, for all $0 <\epsilon \ll 1$, we have
\begin{equation}\label{eq: partial estimates}
||\partial_x(u_1 \cdots u_{k+1})||_{X^{\sigma, \, s,\, -\frac{1}{2}+2\epsilon}} \leq C\prod_{i=1}^{k+1}||u_{i}||_{X^{\sigma, \,s,\frac{1}{2}+2\epsilon}},
\end{equation}
where $C>0$ depends on $s$, $k$ and $\epsilon$.
\end{proposition}
\begin{proof}
The proof is done by applying the inequality $e^{\sigma|\xi|}\leq e^{\sigma|\xi-\xi_{1}-\cdots-\xi_{k+1}|}e^{\sigma|\xi_1|}\cdots e^{\sigma|\xi_{k+1}|}$ and the estimate \eqref{eq:estimate partial derivative} for $e^{\sigma|D_x|}u_i$ in place of $u_i$ for $i=1, ..., k+1$. For a more detailed proof we refer \cite{barostichi2021modified}.
\end{proof}
 
We recall the following well known Strichartz type estimate from \cite{kenig1991osc} (see also Gr\"unrock \cite{grunrock2005bilinear})
\begin{equation}\label{eq:strichartz}
||u||_{L_t^8L_x^8} \leq C||u||_{X^{0,b}}, \text{ for all } b>\frac{1}{2}.
\end{equation}
Moreover, by Sobolev's embedding in the space and time variables, we have the estimate
\begin{equation}\label{eq:sobolev}
||u||_{L_t^\infty L_x^\infty} \leq C ||u||_{X^{\frac{1}{2}+,\frac{1}{2}+}},
\end{equation}
where $\frac{1}{2}+=\frac{1}{2}+\epsilon$.

The following result is immediate using \eqref{eq:strichartz}, \eqref{eq:sobolev} and the generalized Hölder inequality.
\begin{lemma} \label{lemma: L2 norm of product}
For $k \in \N$, $k \geq 4$ and $b >\frac{1}{2}$, we have
\begin{equation}
||u_1 \cdots u_{k+1}||_{L_t^2 L_x^2} \leq C \prod_{i=1}^{k-3}||u_i||_{X^{\frac{1}{2}+,b}}||u_{k-2}||_{X^{0,b}}\cdots ||u_{k+1}||_{X^{0,b}}.
\end{equation}
\end{lemma}

\section{Local well-posedness - Proof of Theorem \ref{teo:local}}\label{sec:3}
In this section we use the estimates from the previous section and provide a proof for the local well-posedness result to the IVP \eqref{eq:gKdV} for given real analytic initial data.

\begin{proof}[Proof of Theorem \ref{teo:local}]
Let $\sigma>0$, $u_0 \in G^{\sigma, \, s}(\R)$ with $s>\frac{k-4}{2k}$. For $0<T\leq 1$, let $\psi_T$ be the cut-off function as defined earlier and consider a solution map given by
\begin{equation*}
\Phi_T(u)=\psi(t)W(t)u_0-\psi_T(t)\int_0^t W(t-t')\frac{\mu}{5}\partial_x(u^{k+1})dt'.
\end{equation*}
Our main goal is to show that there are $b>\frac{1}{2}$, $r>0$ and $T_0=T_0(||u_0||_{G^{\sigma,\,s}})$ such that $\Phi_{T_0}:B(r) \rightarrow B(r)$ is a contraction map, where
\begin{equation*}
B(r)=\left\lbrace u \in X_{T_0}^{\sigma, \, s, \, b}: ||u||_{X^{\sigma, \,s, \, b}}\leq r\right\rbrace.
\end{equation*}

For $u \in B(r)$, applying the linear inequalities \eqref{eq:semigroup 1} and \eqref{eq:semigroup 2} and the estimate \eqref{eq: partial estimates} with $b=\frac{1}{2}+\epsilon$ and $b'=\frac{1}{2}+2\epsilon$ as in the Proposition \ref{prop: estimates}, we obtain
\begin{equation}\label{eq: well defined}
\begin{split}
||\Phi_{T}(u)||_{X^{\sigma, \,s, \, b}} & \leq ||\psi(t)W(t)u_0||_{X^{\sigma, \,s, \, b}}+\left|\left|\psi_T(t)\int_0^t W(t-t')\frac{\mu}{5}\partial_x(u^{k+1})dt'\right|\right|_{X^{\sigma, \,s, \, b}}\\
&\leq C||u_0||_{G^{\sigma, \, s}}+CT^{\frac{1}{a}}||\partial_x(u^{k+1})||_{X^{\sigma, \,s, \, b'-1}} \\
&\leq C||u_0||_{G^{\sigma, \, s}}+CT^{\frac{1}{a}}||u||_{X^{\sigma, \,s, \, b}}^{k+1} \\
&\leq \frac{r}{2}+CT^{\frac{1}{a}}r^{k+1},
\end{split}
\end{equation}
where $\frac{1}{a}=b'-b$ and $r=2C||u_0||_{G^{\sigma, \, s}}$.

By choosing
\begin{equation}\label{eq: condition T0 1}
T_0 \leq \frac{1}{(2Cr^k)^a},
\end{equation}
one gets from \eqref{eq: well defined} that $||\Phi_{T_0}(u)||_{X^{\sigma, \,s, \, b}} \leq r$ for all $u \in B(r)$ showing that $\Phi_{T_0}(B(r))\subset B(r)$.

Now, for $u$, $v \in B(r)$, using \eqref{eq:semigroup 2} once again, we have
\begin{equation*}
||\Phi_{T_0}(u)-\Phi_{T_0}(v)||_{X^{\sigma, \,s, \, b}} \leq CT_{0}^{\frac{1}{a}}||\partial_x(u^{k+1})-\partial_x(v^{k+1})||_{X^{\sigma, \,s, \, b'}}.
\end{equation*}
Since
\begin{equation*}
u^{k+1}-v^{k+1}=(u-v)\sum_{i=0}^{k}u^{i}v^{k-i},
\end{equation*}
from \eqref{eq:estimate partial derivative} we conclude that
\begin{equation}\label{eq-3m}
\begin{split}
||\Phi_{T_0}(u)-\Phi_{T_0}(v)||_{X^{\sigma, \,s, \, b}} &\leq CT_{0}^{\frac{1}{a}}\left|\left|\partial_x\left((u-v)\sum_{i=0}^{k}u^{i}v^{k-i}\right)\right|\right|_{X^{\sigma, \,s, \, b'}}\\
&\leq CT_{0}^{\frac{1}{a}}||u-v||_{X^{\sigma, \,s, \, b}}\sum_{i=0}^{k}||u||_{X^{\sigma, \,s, \, b}}^{i}||v||_{X^{\sigma, \,s, \, b}}^{k-i}\\
&\leq CT_{0}^{\frac{1}{a}}(k+1)r^{k}||u-v||_{X^{\sigma, \,s, \, b}}
\end{split}
\end{equation}
By choosing 
\begin{equation}\label{eq: condition T0 2}
T_0 \leq \frac{1}{(2C(k+1)r^k)^a},
\end{equation}
it follows from \eqref{eq-3m} that
\begin{equation*}
||\Phi_{T_0}(u)-\Phi_{T_0}(v)||_{X^{\sigma, \,s, \, b}} \leq \frac{1}{2}||u-v||_{X^{\sigma, \,s, \, b}}
\end{equation*}
and $\Phi_{T_0}$ is a contraction map.

Finally, it is sufficient to choose $0<T_0\leq 1$ satisfying \eqref{eq: condition T0 1} and \eqref{eq: condition T0 2}. More precisely, considering
\begin{equation}\label{eq:T0}
T_0=\frac{c_0}{(1+||u_0||_{G^{\sigma, \, s}}^2)^{\frac{k}{2}a}}
\end{equation}  for an appropriate constant $c_0>0$ depending on $k$, $s$ and $b$, we conclude that $\Phi_{T_0}$ admits a unique fixed point which is a local solution of the IVP \eqref{eq:gKdV}. Moreover, the solution satisfies 
\begin{equation}\label{eq:bound for solution}
||u||_{X_{T_0}^{\sigma, \,s, \, b}}\leq r=C||u_0||_{G^{\sigma, \, s}}.
\end{equation}

The continuous dependence on the initial data follows by a similar argument and the proof is complete.
\end{proof}

\section{Almost conserved quantity}\label{sec:4}
In order to apply the local result repeatedly to cover arbitrary time interval $[-T, \, T]$, for any $T>0$, we will introduce almost conserved quantity associated to the gKdV equation  \eqref{eq:gKdV} and find some estimates for it. Taking in consideration the conserved quantities \eqref{eq:mass conversation} and \eqref{eq:energy conservation}, we define
\begin{equation}\label{eq:almost conservated}
E_\sigma (t)= \int (e^{\sigma |D_x|}u)^2 dx + \int (\partial_x (e^{\sigma |D_x|}u))^2dx -  \frac{2\mu}{(k+1)(k+2)}\int (e^{\sigma |D_x|}u)^{k+2}dx.
\end{equation}

Note that, for $\sigma=0$, \eqref{eq:almost conservated} turns out to be the conserved quantities \eqref{eq:mass conversation} and \eqref{eq:energy conservation}. However, for $\sigma>0$ this fails to hold and our idea is to find a decay  estimate for $E_\sigma (t)$.   For this purpose,  let $U=e^{\sigma|D_x|}u$ and start by differentiating $E_\sigma (t)$ with respect to $t$, to obtain 

\begin{equation}\label{eq:partial E}
        \frac{d}{dt}(E_\sigma(t))=2\int U U_tdx+2\int \partial_x U \partial_x(\partial_t U) dx - \frac{2\mu}{k+1}\int U^{k+1}\partial_t U dx.
    \end{equation}
    Applying the operator $e^{\sigma|D_x|}$ to the gKdV equation \eqref{eq:gKdV} equation, we get
    \begin{equation}\label{eq:op in gkdv}
        \partial_t U +\partial_x^3 U+\mu U^k\partial_x U=F(U),
        \end{equation}
        where $F(U)$ is given by
\begin{equation}\label{eq:F}
F(U)=\frac{\mu }{k+1}\partial_x\left[U^{k+1}-e^{\sigma |D_x|}\big((e^{-\sigma |D_x|}U)^{k+1}\big)\right].
\end{equation}       
 Using \eqref{eq:op in gkdv} in each term of \eqref{eq:partial E}, one has
\begin{equation*}
    \begin{split}
            \int U\partial_tUdx&=-\int U\partial_x^3Udx-\frac{\mu}{k+2}\partial_x(U^{k+2})dx+\int UF(U)dx,\\
            \int \partial_xU \partial_x(\partial_tU)dx&=-\int\partial_xU\partial_x^4Udx-\mu\int \partial_xU\partial_x(U^k\partial_xU)dx+\int \partial_xU\partial_x(F(U))dx,\\
            \int U^{k+1}\partial_t Udx&=-\int U^{k+1}\partial_x^3Udx-\mu\int U^{k+1}\partial_xUdx+\int U^{k+1}F(U)dx.
    \end{split}
\end{equation*}
We can assume that $U$ and all its partial derivatives tend to zero as $|x| \rightarrow \infty$ (see \cite{selberg2015lower} for a detailed argument) and it follows from integration by parts that
\begin{align}
        \int U\partial_tUdx&=\int UF(U)dx, \label{eq:1a}\\
        \int \partial_xU \partial_x(\partial_tU)dx &=-\frac{\mu}{k+1}\int U^{k+1}\partial_x^3U dx+\int \partial_xU\partial_x(F(U))dx,\label{eq:1b}\\
        \int U^{k+1}\partial_t Udx&=-\int U^{k+1}\partial_x^3Udx+\int U^{k+1}F(U)dx \label{eq:1c}.
\end{align}
Now, inserting \eqref{eq:1a}, \eqref{eq:1b} and \eqref{eq:1c} in \eqref{eq:partial E}, we get
\begin{equation}\label{eq:partial E new}
    \frac{d}{dt}E_\sigma(t)=2\int UF(U)dx + 2\int \partial_xU\partial_x(F(U))dx - \frac{2\mu}{k+1}\int U^{k+1}F(U)dx.
\end{equation}
Integrating \eqref{eq:partial E new} in time over $[0, t']$ for $0<t'\leq T$, we obtain
\begin{equation}\label{eq:estimate E}
    E_\sigma(t') =E_\sigma(0)+R_\sigma(t'),
\end{equation}
where
\begin{equation}\label{eq:R}
\begin{split}
    R_\sigma(t')=& \,\, 2\int \int \chi_{[0,\,t']}UF(U)dxdt + 2\int\int \chi_{[0,\,t']}\partial_xU\partial_x(F(U))dxdt \,\, -\\
    & - \frac{2\mu}{k+1}\int\int \chi_{[0,\,t']}U^{k+1}F(U)dxdt.
    \end{split} 
\end{equation}

In sequel, we find estimates for $F(U)$ in the Bourgain's space norm.

\begin{lemma}\label{lemma: estimates F}
Let $F$ be as defined in \eqref{eq:F} and $\sigma >0$. Then, for any $b=\frac{1}{2}+\epsilon$, $0<\epsilon \ll 1$, and for all $\alpha \in \left[0,\frac{k+4}{2k}\right)$,
\begin{equation}\label{eq: estimate F in L2}
||F(U)||_{L_x^2 L_t^2} \leq C \sigma^{\alpha} ||U||_{X^{1, \, b}}^{k+1},
\end{equation}
\begin{equation}\label{eq: estimate F in X}
||\partial_x(F(U))||_{X^{0, \, b-1}} \leq C \sigma^{\alpha} ||U||_{X^{1, \, b}}^{k+1}
\end{equation}
for some constant $C>0$ independent on $\sigma$.
\end{lemma}
\begin{proof}
Observe that
\begin{equation}\label{eq: transf F}
|\widehat{F(U)}(\xi, \tau)| \leq C|\xi|\int_{*} \big(1-e^{-\sigma(|\xi_1|+\cdots+|\xi_{k+1}|-|\xi|)}\big) |\widehat{U}(\xi_1,\tau_1)| \cdots |\widehat{U}(\xi_{k+1},\tau_{k+1})|,
\end{equation}
where $*$ denotes the integral over the set $\xi=\xi_1 + \cdots +\xi_{k+1}$ and $\tau=\tau_1 + \cdots +\tau_{k+1}$.
 
Now, from the inequality
\begin{equation*}
    e^x-1\leq x^{\alpha}e^x, \quad \forall \, x\geq 0 \text{ and } \alpha \in [0,\, 1],
\end{equation*}
we get
\begin{equation}\label{eq: 1-e}
    1-e^{-\sigma(|\xi_1|+\cdots+|\xi_{k+1}|-|\xi|)} \leq \sigma^{\alpha}(|\xi_1|+\cdots+|\xi_{k+1}|-|\xi|)^{\alpha}.
\end{equation}
Without loss of generality, we can assume that $|\xi_1| \leq \cdots \leq |\xi_{k+1}|$. A simple calculation shows
\begin{equation*}
    |\xi_1|+\cdots+|\xi_{k+1}|-|\xi| \leq 2k|\xi_{k}|,
\end{equation*}
and consequently, the estimate \eqref{eq: 1-e} yields
\begin{equation}\label{eq: 1-e nova}
     1-e^{-\sigma(|\xi_1|+\cdots+|\xi_{k+1}|-|\xi|)} \leq C\sigma^{\alpha}|\xi_{k}|^{\alpha}.
\end{equation}
Using \eqref{eq: 1-e nova} in \eqref{eq: transf F}, we have
\begin{equation}\label{eq: tranf F nova}
|\widehat{F(U)}(\xi, \tau)| \leq C\sigma^{\alpha}|\xi|\int_{*} |\xi_k|^{\alpha}|\widehat{U}(\xi_1,\tau_1)| \cdots |\widehat{U}(\xi_{k+1},\tau_{k+1})|.
\end{equation}
Moreover, one has
\begin{equation}\label{eq:4.19}
    |\xi||\xi_k|^{\alpha}\leq (k+1)|\xi_{k+1}||\xi_k|^{\alpha}.
\end{equation}
Now, using \eqref{eq:4.19} in \eqref{eq: tranf F nova} and an use of Plancherel's identity implies that
\begin{equation}\label{eq-420}
\begin{split}
  \!\!\!\!\!\!\!  ||F(U)||_{L_x^2 L_t^2}&\leq C \sigma^{\alpha} \left|\left| \int_{*} |\widehat{U}(\xi_1,\tau_1)| \cdots |\widehat{U}(\xi_{k-1},\tau_{k-1})||\widehat{D_x^{\alpha} U}(\xi_{k},\tau_{k})||\widehat{D_x U}(\xi_{k+1},\tau_{k+1})|\right|\right|_{L_{\xi}^2 L_{\tau}^2}\\
    &=C\sigma^{\alpha} ||w_1 \cdots w_{k+1}||_{L_x^2 L_t^2},
\end{split}
\end{equation}
where $\widehat{w_1}(\xi, \, \tau)=\cdots=\widehat{w_{k-1}}(\xi, \, \tau)=|\widehat{U}(\xi,\tau)|$, $\widehat{w_{k}}(\xi, \, \tau)=|\widehat{D_x^{\alpha}U}(\xi,\tau)|$ and $\widehat{w_{k+1}}(\xi, \, \tau)=|\widehat{D_xU}(\xi,\tau)|$. 

By applying Lemma \ref{lemma: L2 norm of product}  it follows from \eqref{eq-420} that
\begin{equation*}
\begin{split}
    ||F(U)||_{L_x^2 L_t^2} &\leq C{\sigma}^{\alpha} \prod_{i=1}^{k-3}||w_i||_{X^{\frac{1}{2}+,b}}||w_{k-2}||_{X^{0,b}}\cdots ||w_{k+1}||_{X^{0,b}}\\
    &=C{\sigma}^{\alpha} \left( \prod_{i=1}^{k-3}||U||_{X^{\frac{1}{2}+,b}}\right)||U||^2_{X^{0,b}}||D_x^{\alpha}U||_{X^{0,b}}||D_xU||_{X^{0,b}}\\
    &\leq C{\sigma}^{\alpha} ||U||_{X^{1,b}}^{k+1}.
\end{split}
\end{equation*}
This completes the proof of \eqref{eq: estimate F in L2}.

To prove \eqref{eq: estimate F in X}, first observe that for $0 \leq j \leq 1$, one has $\langle \xi \rangle^{-j} \leq |\xi|^{-j}$ for all $\xi \neq 0$. Using this fact, we obtain
\begin{equation}\label{eq: partial F}
\begin{split}
||\partial_xF(U)||_{X^{0,\, b-1}} &=||\langle \tau - \xi^3\rangle^{b-1} |\xi| |\widehat{F(U)}(\xi,\tau)| ||_{L_{\xi}^2 L_{\tau}^2}\\
&\leq ||\langle \tau - \xi^3 \rangle^{b-1}\langle\xi \rangle^j |\xi|^{1-j}|\widehat{F(U)}(\xi,\tau)|||_{L_{\xi}^2 L_{\tau}^2}.
\end{split}
\end{equation}
Assuming $|\xi_1|\leq \cdots \leq |\xi_{k+1}|$, one has $|\xi|^{1-j} \leq (k+1)^{1-j}|\xi_{k+1}|^{1-j}$ and using  \eqref{eq: tranf F nova}, the estimate \eqref{eq: partial F} yields
\begin{equation}\label{eq:4.23}
    \begin{split}
||\partial_xF(U)||_{X^{0,\, b-1}}  &\leq C\sigma^{\alpha}  \left|\left| \langle \tau - \xi^3 \rangle^{b-1}\langle\xi \rangle^j |\xi|^{1-j}  |\xi|\int_{*} |\xi_{k}|^{\alpha}|\widehat{U}(\xi_1,\tau_1)| \cdots |\widehat{U}(\xi_{k+1},\tau_{k+1})|\right|\right|_{L_{\xi}^2 L_{\tau}^2}\\
&\leq C\sigma^{\alpha}  \left|\left| \langle \tau - \xi^3 \rangle^{b-1}\langle\xi \rangle^j  |\xi|\int_{*} |\xi_{k}|^{\alpha}|\xi_{k+1}|^{1-j}|\widehat{U}(\xi_1,\tau_1)| \cdots |\widehat{U}(\xi_{k+1},\tau_{k+1})|\right|\right|_{L_{\xi}^2 L_{\tau}^2}\\
&=C\sigma^{\alpha} ||\partial_x(v_1 \cdots v_{k+1})||_{X^{j, \, b-1}},
\end{split}
\end{equation}
for $v_1, ..., v_{k+1}$ defined by $\widehat{v_1}(\xi, \, \tau)=\cdots=\widehat{v_{k-1}}(\xi, \, \tau)=|\widehat{U}(\xi,\tau)|$, $\widehat{v_{k}}(\xi, \, \tau)=|\widehat{D_x^{\alpha}U}(\xi,\tau)|$ and $\widehat{v_{k+1}}(\xi, \, \tau)=|\widehat{D_x^{1-j}U}(\xi,\tau)|$.

For every $j>\frac{k-4}{2k}$, we can use the estimate \eqref{eq: partial estimates} with $b=\frac{1}{2}+\epsilon$, to obtain from \eqref{eq:4.23} that
\begin{equation}\label{eq: partial F final}
\begin{split}
    ||\partial_xF(U)||_{X^{0,\, b-1}} &\leq C\sigma^{\alpha}||v_1||_{X^{j,\,b}}\cdots ||v_{k+1}||_{X^{j,\,b}}\\    &=C\sigma^{\alpha}\left(\prod_{i=1}^{k-1}||U||_{X^{j,\,b}}\right) ||D_x^{\alpha}U||_{X^{j,\,b}}||D_x^{1-j}U||_{X^{j,\,b}}\\  &\leq C\sigma^{\alpha}\left(\prod_{i=1}^{k-1}||U||_{X^{1,\,b}}\right) ||U||_{X^{j+\alpha,\,b}}||U||_{X^{1,\,b}}.\\
    \end{split}
\end{equation}
Now, we choose $j>\frac{k-4}{2k}$ and $\alpha \in [0,1]$ such that $j+\alpha\leq 1$. Indeed, for $\delta>0$ arbitrarialy small, considering $j=\frac{k-4}{2k}+\delta$ we get $j+\alpha\leq 1$ if we choose $\alpha \leq \frac{k+4}{2k}- \delta$. Therefore, for $0<\alpha \leq \frac{k+4}{2k}- \delta$, from \eqref{eq: partial F final} one has
\begin{equation*}
    ||\partial_xF(U)||_{X^{0,\, b-1}} \leq C\sigma^{\alpha} ||U||_{X^{1,b}}^{k+1},
\end{equation*}
as desired.
\end{proof}

In the following we use the estimate obtained in Lemma \ref{lemma: estimates F} to prove that the quantity $E_{\sigma}(t)$ defined in \eqref{eq:almost conservated} is almost conserved.

\begin{proposition}\label{prop:conserved quantity}
    Let $\sigma>0$ and $\alpha \in [0, \, \frac{k+4}{2k})$. Then for any $b=\frac{1}{2}+\epsilon$, $0<\epsilon \ll 1$, there exists $C>0$ such that for any solution $u \in X_T^{\sigma, \,1, \, b}$ to the IVP \eqref{eq:gKdV} in the interval $[0,\,T]$, we have
\begin{equation}\label{eq:conserved quantity}
    \sup_{t\in [0,\,T]} E_\sigma (t) \leq E_\sigma (0) + C\sigma^{\alpha}||u||_{X_T^{\sigma, \,1, \, b}}^{k+2}(1+||u||_{X_T^{\sigma, \,1, \, b}}^{k}).
\end{equation}
\end{proposition}
\begin{proof}
   Recalling \eqref{eq:estimate E} and \eqref{eq:R}, we have to estimate each term of $|R_\sigma(t')|$ for all $0<t'\leq T$. For the first and the third terms in \eqref{eq:R} we use the Cauchy-Schwarz inequality, Lemmas \ref{lemma:restriction} and \ref{lemma:estimates partial} and the estimate \eqref{eq: estimate F in L2} restricted to the time slab and obtain that for any $b=\frac{1}{2}+\epsilon$, $0<\epsilon \ll 1$, there exists $C>0$ such that
\begin{equation}\label{eq:estimate first term}
    \begin{split}
    \left|\int \int \chi_{[0,\,t']}UF(U)dxdt\right| &\leq ||\chi_{[0,\,t']}U||_{L_x^2L_t^2}||\chi_{[0,\,t']}F(U)||_{L_x^2L_t^2}\\
    &\leq ||U||_{X_T^{0,\,0}}||F(U)||_{X_T^{0,\,0}}\\
    &\leq C\sigma^{\alpha}||u||_{X_T^{ \sigma, \, 1, \, b}}^{k+2}
    \end{split}
\end{equation}
and 
\begin{equation}\label{eq:estimate third term}
    \begin{split}
    \left|\int \int \chi_{[0,\,t']}U^{k+1}F(U)dxdt\right| &\leq ||\chi_{[0,\,t']}U^{k+1}||_{L_x^2L_t^2}||\chi_{[0,\,t']}F(U)||_{L_x^2L_t^2}\\
    &\leq ||U^{k+1}||_{X_T^{0,\,0}}||F(U)||_{X_T^{0,\,0}}\\
    &\leq C\sigma^{\alpha}||u||_{X_T^{ \sigma, \, 1, \, b}}^{2(k+1)},
    \end{split}
\end{equation}
for all $0<t'\leq T$.

For the second term in \eqref{eq:R}, we use the Cauchy-Schwarz inequality, Lemma \ref{lemma:restriction} and estimate \eqref{eq: estimate F in X}, to obtain
\begin{equation}\label{eq:estimate second term}
    \begin{split}
    \left|\int \int \chi_{[0,\,t']}\partial_xU\partial_x(F(U))dxdt\right|& \leq ||\chi_{[0,\,t']}\partial_xU||_{X^{0, \, 1-b}}||\chi_{[0,\,t']}\partial_x(F(U))||_{X^{0, \, b-1}}\\
    & \leq ||\partial_xU||_{X_T^{0, \, 1-b}}||\partial_x(F(U))||_{X_T^{0, \, b-1}}\\
    &\leq C\sigma^{\alpha}||u||_{X_T^{ \sigma, \, 1, \, b}}^{k+2},
    \end{split}
\end{equation}
 for all $0<t'\leq T$, where $\frac{1}{2}<b<1$ is the same as before.

 The result follows combining the estimates \eqref{eq:estimate first term}, \eqref{eq:estimate third term} and \eqref{eq:estimate second term}.
\end{proof}

\begin{corollary}\label{cor: almost conserved quantity}
Let $\sigma>0$, $k$ even and $\alpha \in [0,\frac{k+2}{2k})$. Then there exists $C>0$ such that the solution $u \in X_T^{ \sigma, \, 1, \, b}$ to the IVP \eqref{eq:gKdV} in the defocusing case ($\mu=-1$) given by Theorem \ref{teo:local}, we have
\begin{equation}\label{eq:estimate energy}
    \sup_{t\in [0,\,T]} E_\sigma (t) \leq E_\sigma (0) + C\sigma^{\alpha} E_\sigma (0)^{\frac{k}{2}+1}(1+E_\sigma (0)^{\frac{k}{2}}).
\end{equation}
\end{corollary}
\begin{proof}
    For $\mu=-1$, from \eqref{eq:almost conservated}, since $k$ is even, we have
    \begin{equation}\label{eq:conserved defocusing}
    E_\sigma (0)= ||u_0||_{G^{\sigma, \, 1}}^2+\frac{2}{(k+1)(k+2)}||e^{\sigma |D_x|}u_0||_{L^{k+2}_x}^{k+2} \geq ||u_0||_{G^{\sigma, \, 1}}^2.
\end{equation}
From \eqref{eq:bound for solution} and \eqref{eq:conserved defocusing}, we get
\begin{equation}\label{eq:estimate u energy}
    ||u||_{X_T^{\sigma, \,1, \, b}}\leq C E_\sigma (0)^{\frac{1}{2}}.
\end{equation}
The required estimate follows from \eqref{eq:conserved quantity} and the estimate \eqref{eq:estimate u energy}.
\end{proof}

\begin{remark} For the solutions to the IVP \eqref{eq:gKdV} with $k$ odd we do not know the sign of $\int (e^{\sigma |D_x|}u)^{k+2}dx$ and it is not possible to have an inequality of the form \eqref{eq:conserved defocusing}.

For the solutions to the IVP \eqref{eq:gKdV} in the focusing case ($\mu=1$) with $k$ even, from \eqref{eq:almost conservated}, one has
\begin{equation}
    E_\sigma (0)= ||u_0||_{G^{\sigma, \, 1}}^2-\frac{2}{(k+1)(k+2)}||e^{\sigma |D_x|}u_0||_{L^{k+2}_x}^{k+2} \leq ||u_0||_{G^{\sigma, \, 1}}^2,
\end{equation}
This estimate cannot be used to obtain an estimate of the form \eqref{eq:estimate energy} which plays a crucial role in the argument used to prove the Theorem \ref{teo:global}. 

For this reason, we considered only the defocusing case with $k$ even to obtain the decay rate of the radius of spatial analyticity to the gKdV equation \eqref{eq:gKdV}. 
\end{remark}

\section{Global Analytic Solution - Proof of Theorem \ref{teo:global}}\label{sec:5}

Now we are in position to provide the proof of the main result of this work.

\begin{proof}[Proof of Theorem \ref{teo:global}]
Fix $\sigma_0>0$, $\mu=-1$, $k\geq 4$ even, $s>\frac{k-4}{2k}$ and $u_0 \in G^{\sigma_0, \, s}(\R)$. Moreover, let $\alpha \in [0, \frac{k-4}{2k})$ and let $b \in (\frac{1}{2},1)$ be as in Corollary \ref{cor: almost conserved quantity}. By invariance of the gKdV equation under reflection $(t,\,x) \rightarrow (-t,\,-x)$, it suffices to consider $t \geq 0$. Thus, we have to prove that the local solution $u$ given by the Theorem \ref{teo:local} can be extended to any time interval $[0,\,T]$ and satisfies
\begin{equation*}
    u \in C([0,\,T]:G^{\sigma(T),\,s}) \quad \text{for all } T>0,
\end{equation*}
where
\begin{equation}\label{eq:sigma T}
\sigma(T)\geq cT^{-\frac{1}{\alpha}}
\end{equation}
and $c>0$ is a constant depending on $\|u_0\|_{G^{\sigma_0, \, s}}, \, \sigma_0, \, k, \, s$ and $\alpha$.

By Theorem \ref{teo:local}, there is a maximal time $T^*=T^*(\|u_0\|_{G^{\sigma_0, \, s}},\, \sigma_0, \, k ,\, s) \in (0,\, \infty]$ such that
\begin{equation*}
   u \in C([0,\,T^*):G^{\sigma_0,\,s}).
\end{equation*}
If $T^*=\infty$, we are done. We assume that $T^*<\infty$ and in this case it remains to prove 
\begin{equation*}
    u \in C([0,\,T]:G^{\sigma(T),\,s}) \quad \text{for all } T\geq T^*.
\end{equation*}
If we prove this in the case $s=1$ then  the general case will essentially reduces to $s=1$ using the inclusion \eqref{eq:Gegrey embedding}. For more details we refer the work \cite{selberg2015lower}.

Assume $s=1$. From \eqref{eq:conserved defocusing}, we have $E_{\sigma_0}(0) \geq ||u_0||_{G^{\sigma_0,\,1}}^2$ and from Theorem \ref{teo:local} with $s=1$, we can take the lifespan $T_0$ given in \eqref{eq:T0} as
\begin{equation*}
    T_0=\frac{c_0}{(1+E_{\sigma_0}(0))^a}
\end{equation*}
for appropriate constants $c_0>0$ and $a>1$.

Let $T\geq T^*$. We will show that, for $\sigma>0$ sufficiently small,
\begin{equation}\label{eq:objective energy}
    E_\sigma(t) \leq 2E_{\sigma_0}(0) \quad \text{for } t \in [0,\,T].
\end{equation}
The desired result will follow since, in this case, $||u(t)||_{G^{\sigma, \, 1}} \leq E_{\sigma}(t)$ and
\begin{equation*}
    \begin{split}
    E_{\sigma_0} (0)&= ||u_0||_{G^{\sigma_0, \, 1}}^2+\frac{2}{(k+1)(k+2)}||e^{\sigma_0 |D_x|}u_0||_{L^{k+2}}^{k+2}\\
    &\leq ||u_0||_{G^{\sigma_0, \, 1}}^2+C||D_x(e^{\sigma_0|D_x|}u_0)||_{L^{2}}^{\frac{k}{2}}||e^{\sigma_0|D_x|}u_0||_{L^{2}}^{\frac{k}{2}+2}\\
    &\leq ||u_0||_{G^{\sigma_0, \,1}}^2+C||u_0||_{G^{\sigma_0,\,1}}^{k+2}<\infty,
    \end{split}
\end{equation*}
where we used Gagliardo-Niremberg inequality.

To prove \eqref{eq:objective energy}, we will use repeatedly Theorem \ref{teo:local} and Corollary \ref{cor: almost conserved quantity} with the time step
\begin{equation}\label{eq:delta}
    \delta=\frac{c_0}{(1+2E_{\sigma_0}(0))^a}.
\end{equation}
Since $\delta\leq T_0 \leq T^* \leq T$, it follows that there exists $n \in \N$ such that $T \in [n\delta,\,(n+1)\delta)$ and by induction, we will show that for $j \in \{1, \cdots,\, n\}$,
\begin{align}
    \sup_{t \in [0,\, j\delta]} E_\sigma (t) &\leq E_{\sigma}(0)+2^{k+1}C\sigma^{\alpha} jE_{\sigma_0}(0)^{\frac{k}{2}+1}(1+E_{\sigma_0}(0)^{\frac{k}{2}}),\label{eq:induction 1}\\ 
    \sup_{t \in [0,\, j\delta]} E_\sigma (t) &\leq 2 E_{\sigma_0} (0), \label{eq:induction 2}
\end{align}
under the smallness conditions
\begin{equation}\label{eq: condition 1}
    \sigma \leq \sigma_0
\end{equation}
and
\begin{equation}\label{eq: condition 2}
    2^{k+2}\frac{T}{\delta}C\sigma^{\alpha}E_{\sigma_0}(0)^\frac{k}{2}(1+E_{\sigma_0}(0)^\frac{k}{2}) \leq 1,
\end{equation}
where $C>0$ is the constant in Corollary \ref{cor: almost conserved quantity}.

In the first step, we cover the interval $[0,\, \delta]$ and by Corollary \ref{cor: almost conserved quantity}, we have
\begin{equation}\label{eq:indiction base}
\begin{split}
    \sup_{t\in [0,\,\delta]} E_\sigma (t) &\leq E_\sigma (0) + C\sigma^{\alpha} E_\sigma (0)^{\frac{k}{2}+1}(1+E_\sigma (0)^{\frac{k}{2}})\\
    &\leq E_\sigma (0) + C\sigma^{\alpha} E_{\sigma_0} (0)^{\frac{k}{2}+1}(1+E_{\sigma_0} (0)^{\frac{k}{2}}).
    \end{split}
\end{equation}
Using the conditions \eqref{eq: condition 1} and \eqref{eq: condition 2} in \eqref{eq:indiction base} we conclude that
\begin{equation*}
    \begin{split}
    \sup_{t \in [0,\, \delta]} E_\sigma (t) &\leq E_{\sigma} (0) + 2^{k+2}\frac{T}{\delta}C\sigma^{\alpha} E_{\sigma_0} (0)^{\frac{k}{2}+1}(1+E_{\sigma_0} (0)^{\frac{k}{2}})\\
    &\leq 2 E_{\sigma_0} (0).
    \end{split}
\end{equation*}
Now, assume that \eqref{eq:induction 1} and \eqref{eq:induction 2} hold for some $j \in \{1,\cdots,\, n\}$. For $j+1$, applying the local well posedness result with initial data $u(j\delta)$ and the estimate \eqref{eq:conserved defocusing}, we have
\begin{equation}\label{eq:induction passage}
\begin{split}
\sup_{t \in [j\delta, (j+1)\delta]}E_\sigma (t)&\leq E_\sigma(j\delta) + C\sigma^{\alpha} E_\sigma (j\delta)^{\frac{k}{2}+1}(1+E_\sigma (j\delta)^{\frac{k}{2}})\\
&\leq E_\sigma(j\delta) + 2^{k+1}C\sigma^{\alpha} E_{\sigma_0} (0)^{\frac{k}{2}+1}(1+E_{\sigma_0} (0)^{\frac{k}{2}})\\
&\leq E_\sigma(0) + 2^{k+1}C\sigma^{\alpha}(j+1)E_{\sigma_0} (0)^{\frac{k}{2}+1}(1+E_{\sigma_0} (0)^{\frac{k}{2}}).
\end{split}
\end{equation}
Moreover, since
\begin{equation*}
    j+1 \leq n+1 \leq \frac{T}{\delta}+1 \leq \frac{2T  }{\delta},
\end{equation*}
it follows from the condition \eqref{eq: condition 2} and \eqref{eq:induction passage} that
\begin{equation*}
    \begin{split}
    \sup_{t \in [j\delta,\,(j+1)\delta]}  E_\sigma (t)&\leq E_\sigma(0) + 2^{k+2}C\sigma^{\alpha}\frac{2T  }{\delta}E_{\sigma_0} (0)^{\frac{k}{2}+1}(1+E_{\sigma_0} (0)^{\frac{k}{2}})\\
    &\leq 2 E_{\sigma_0}(0).
    \end{split}
\end{equation*}

Thus, we proved \eqref{eq:objective energy} under the assumptions \eqref{eq: condition 1} and \eqref{eq: condition 2}. Since $T\geq T^*$, the condition \eqref{eq: condition 2} must fail for $\sigma=\sigma_0$ since otherwise we would be able to continue the solution in $G^{\sigma_0, \, 1}$ beyond the time $T$, contradicting the maximality of $T^*$.  Therefore, there is some $\sigma \in (0,\,\sigma_0)$ such that 
\begin{equation}\label{eq:sigma}
    2^{k+2}\frac{T}{\delta}C\sigma^{\alpha}E_{\sigma_0}(0)^\frac{k}{2}(1+E_{\sigma_0}(0)^\frac{k}{2}) = 1.
\end{equation}
and hence
\begin{equation*}
    \sigma(T)=\left[\frac{\delta}{2^{k+2}TCE_{\sigma_0}(0)^\frac{k}{2}(1+E_{\sigma_0}(0)^\frac{k}{2})}\right]^{\frac{1}{\alpha}}=: c_1T^{-\frac{1}{\alpha}},
\end{equation*}
which gives \eqref{eq:sigma T} if we choose $c\leq c_1$.

From Corollary \ref{cor: almost conserved quantity}, one can consider $\alpha \in [0,\, \frac{k+4}{2k})$. Choosing the maximum value of $\alpha$, one has
\begin{equation*}
    \sigma(T) \geq cT^{-\left(\frac{2k}{k+4}+\epsilon\right)}
\end{equation*}
for $\epsilon>0$ arbitrarily small and the proof for $s=1$ is concluded. For other values of $s \in \R$, the proof follows using the inclusion \eqref{eq:Gegrey embedding} as described above.
\end{proof}

\begin{remark} There are several works (see for example \cite{HHP}, \cite{HG}, \cite{QL}) in which the authors consider the gKdV equation in the periodic setting and obtained the local well-posedness results for given real analytic initial data. Recently, Himonas et al. in \cite{HKS} considered the dispersion generalized KdV equation in the periodic setting and constructed almost conserved quantity at the $L^2$-level of Sobolev regularity to obtain an algebraic lower bound for radius of spatial analyticity for the global solution emanating from the real analytic initial data. This motivates one to initiate a similar study for  the gKdV equation \eqref{eq:gKdV}  with $k\geq 2$ as well. However, as discussed in the previous section, in the periodic case too it would be necessary to construct almost conserved quantity at the higher lever of Sobolev regularity. The multilinear estimates in the periodic setting are more delicate and that requires a considerable piece of work.   This will be one of the topics that we intend to address in future projects.
\end{remark}

\end{document}